\documentclass[12pt]{amsart}

\usepackage{amsmath,amssymb,esint,amsthm,bbm}
 \usepackage{xcolor}

\usepackage[colorlinks=true, pdfstartview=FitV, linkcolor=blue,
citecolor=blue, urlcolor=blue]{hyperref}

\calclayout

\catcode`@=11 \@addtoreset{equation}{section} \catcode`@=12

\DeclareMathOperator{\supp}{supp}

\DeclareMathOperator*{\esup}{\text{ess\,sup }}

\newcommand{\sgn}{\text{sgn}}
\newcommand{\ran}{\rangle}
\newcommand{\lan}{\langle}
\newcommand{\vv}{{\vec{\bf v}}}
\newcommand{\ww}{{\vec{\bf w}}}
\newcommand{\uu}{\vec{\bf u}}

\newcommand{\vphi}{\varphi}
\newcommand{\ma}{\mathcal{A}}
\newcommand{\mx}{\mathcal{X}}

\newcommand{\qform}{{\mathcal{Q}}}

\newcommand{\field}[1]{\mathbf{#1}}
\newcommand{\R}{\field{R}}

\newcommand{\bea}{\begin{eqnarray}}
\newcommand{\cend} {\end{center}}

\newcommand{\eea}{\end{eqnarray}}

\newcommand{\disp}{\displaystyle}

\newcommand{\ra}{\rightarrow}

\newcommand{\beginc} {\begin{center}}

\newcommand{\Sn}{\mathcal{S}}

\newtheorem{thm}{\textbf{Theorem}}[section]
\newtheorem{lem}[thm]{\textbf{Lemma}}
\newtheorem{pro}[thm]{\textbf{Proposition}}
\newtheorem{rem}[thm]{\textbf{Remark}}

\newtheorem{defn}[thm]{\textbf{Definition}}
\theoremstyle{remark}

\theoremstyle{definition}
\newtheoremstyle{Claim}{}{}{\itshape}{}{\itshape\bfseries}{:}{ }{#1}
\theoremstyle{Claim}

\title{Global Sobolev
  inequalities and Degenerate P-Laplacian equations} 

\author{David Cruz-Uribe, OFS, Scott Rodney, and Emily
  Rosta}

\address{David Cruz-Uribe, OFS \\
Dept. of Mathematics \\
University of Alabama \\
 Tuscaloosa, AL 35487, USA}
\email{dcruzuribe@ua.edu}

\address{Scott Rodney\\
Dept. of Mathematics, Physics and Geology \\ 
Cape Breton University \\
Sydney, NS B1Y3V3, CA} 
\email{scott\_rodney@cbu.ca}

\thanks{D.~Cruz-Uribe is supported by NSF Grant DMS-1362425 and
  research funds from the Dean of the College of Arts \& Sciences, the
  University of Alabama. S.~Rodney is supported by the NSERC Discovery
  Grant program.  E.~Rosta is a graduate of the undergraduate honours
  program in mathematics at Cape Breton University; her work was
  supported through the NSERC USRA program.}

\date{December 28, 2017}

\keywords{degenerate Sobolev spaces, Sobolev inequality,
  $p$-Laplacian}

\subjclass{35B65,35J70,42B35,42B37,46E35}

\begin{document}

\begin{abstract} 
  We prove that  a local, weak Sobolev inequality implies a global
  Sobolev estimate using
  existence and regularity results for a family of
  $p$-Laplacian equations.   Given $\Omega\subset\mathbb{R}^n$, let $\rho$ be a
  quasi-metric on $\Omega$, and let $Q$ be an $n\times n$ semi-definite matrix function
  defined on $\Omega$.   
For an open set $\Theta\Subset\Omega$, we give sufficient conditions
to show that if  the local
weak   Sobolev inequality 
\[  \Big(\fint_B
  |f|^{p\sigma}dx\Big)^\frac{1}{p\sigma} \leq C\Big[ r(B)\fint_B
  |\sqrt{Q}\nabla f|^pdx + \fint_B
  |f|^pdx\Big]^\frac{1}{p} \]
holds for some $\sigma>1$, all balls $B\subset \Theta$,  and functions $f\in
Lip_0(\Theta)$, then the global Sobolev inequality
\[  \Big(\int_\Theta
  |f|^{p\sigma}dx\Big)^\frac{1}{p\sigma} \leq C\Big( \int_\Theta
  |\sqrt{Q}\nabla f(x)|^pdx\Big)^\frac{1}{p} \]
also holds.   Central to our proof is showing the existence and
boundedness of solutions of the Dirichlet problem 
\[
\begin{cases}
\mx_{p,\tau} u &  = \varphi  \text{ in } \Theta \\
u  & = 0  \text{ in } \partial \Theta,
\end{cases}
\]
where $\mx_{p,\tau}$ is a degenerate $p$-Laplacian operator with a
zero order term:
\[
\mx_{p,\tau} u 
= \text{div}\Big(\big|\sqrt{Q} \nabla u\big|^{p-2}Q\nabla u\Big)
- \tau |u|^{p-2}u.
\]
\end{abstract}

\maketitle


\section{Introduction}

Given an open set $\Theta \subset \R^n$, the classical Sobolev inequality,
\[ \bigg(\int_\Theta
|f|^{p\sigma}dx\bigg)^\frac{1}{p\sigma} \leq C\bigg( \int_\Theta
|\nabla f|^pdx\bigg)^\frac{1}{p}, \]
holds for $1\leq p<n$, $\sigma=\frac{n}{n-p}>1$, and all functions $f\in Lip_0(\Theta)$
(that is, Lipschitz function such that $\supp(f) \Subset \Omega$).
For this result and extensive generalizations, see \cite{E,GT,HK}.

We are interested in determining sufficient conditions for a degenerate
Sobolev inequality,
\begin{equation} \label{eqn:global-sobolev}
 \bigg(\int_\Theta
|f|^{p\sigma}\,dx\bigg)^\frac{1}{p\sigma} \leq C\bigg( \int_\Theta
|\sqrt{Q}\nabla f|^p\,dx\bigg)^\frac{1}{p}, 
\end{equation}
to hold, where $1<p<\infty$, $\sigma>1$, and $Q$ is an $n\times n$
matrix of measurable functions defined on $\Theta$ such that for
almost every $x\in \Theta$, $Q(x)$ is semi-definite .  Such
inequalities arise naturally in the study of degenerate elliptic PDEs:
a global Sobolev inequality is necessary to prove the existence of
weak solutions (see, for instance~\cite{CMN,MR}) and to prove compact
embeddings of (degenerate) Sobolev spaces (see \cite{CRW}).

Our goal is to show that such global estimates can be derived from 
weaker, local Sobolev inequalities.  

\begin{defn}\label{sobprop} 
  Given $1\leq p<\infty$ and $\sigma>1$, a local Sobolev
  property of order $p$ with gain $\sigma$ holds in $\Omega$ if there
  is a constant $C_0>0$ and a positive, continuous function $r_1 : \Omega
  \rightarrow (0,\infty)$ such that for any $y\in \Omega$,
  $0<r<r_1(y)$, and $f\in Lip_0(B(y,r))$,
\begin{equation}\label{localsob} 
\bigg(\fint_{B(y,r)}
  |f|^{p\sigma}\,dx\bigg)^\frac{1}{p\sigma} 
\leq
  C_0r\bigg(\fint_{B(y,r)} |\sqrt{Q}\nabla
  f|^p\,dx\bigg)^\frac{1}{p} 
+ 
C_0\bigg(\fint_{B(y,r)}
  |f|^p\,dx\bigg)^\frac{1}{p}.
\end{equation}
\end{defn}

Local Sobolev estimates arise naturally in the study of regularity of
degenerate elliptic equations (see~\cite{SW1,SW2}), but they are not
sufficient for proving the existence of solutions.  So it is natural
to ask if local inequalities imply global ones.   The obvious
approach is to use a partition of unity argument, but this does not
work. 
If the local Sobolev property of order $p$ with gain $\sigma$ holds on
$\Omega$, then given any open set $\Theta\Subset \Omega$ a partition of unity argument shows that there is a constant $C(\Theta)$ such that 
\begin{equation} \label{weakglobalsob} 
\bigg(\int_\Theta |f|^{p\sigma}\,dx\bigg)^\frac{1}{p\sigma} 
\leq
 C(\Theta)\bigg[\bigg(\int_\Theta |\sqrt{Q}\nabla f|^p\,dx\bigg)^\frac{1}{p} 
+ 
\bigg(\int_\Theta |f|^pdx\bigg)^\frac{1}{p}\bigg]
\end{equation}
holds for every $f\in Lip_0(\Theta)$.  However, we cannot remove the
second term on the right of \eqref{weakglobalsob} even when the second
term on the right of \eqref{localsob} is not present.

Nevertheless, with some additional assumptions we are able to pass
from a local to a global Sobolev inequality.  To set the stage for our
main result, we first state an important special case.

\begin{thm} \label{thm:main-special}
Fix $1<p<\infty$ and let $\Omega \subset \mathbb R^n$ be an open set.
Suppose that $Q$ is a semi-definite matrix function in
$L^\infty(\Omega)$.
Suppose that the local Sobolev property of order $p$ with gain
$\sigma>1$~\eqref{localsob} holds, and suppose further that a local
Poincar\'e inequality
\[ \bigg(\fint_{B(y,r)} |f-f_{B(y,r)}|^p\,dx\bigg)^{\frac{1}{p}}
\leq Cr \bigg(\fint_{B(y,\beta r)} |\sqrt{Q}\nabla
f|^p\,dx\bigg)^{\frac{1}{p}} \]
holds for $y\in \Omega$, $\beta\geq 1$, $0<\beta r<r_1(y)$, and $f\in
Lip_0(\Omega)$.  Then given any open set $\Theta \Subset \Omega$,
there exists a constant $C(\Theta)$ such that
\eqref{eqn:global-sobolev} holds.
\end{thm}

\medskip

Our main result, Theorem~\ref{mainlocglob}, generalizes
Theorem~\ref{thm:main-special} in several ways.  First, we remove the
assumption that $Q$ is bounded, allowing it to be singular as well as
degenerate.  Second, we can change the underlying geometry by
replacing the Euclidean metric with a quasi-metric, and defining balls
with respect to this metric.   The statement is rather technical and
requires some additional hypotheses, which is why we have deferred the
statement until below.

The remainder of the paper is organized as follows.  In
Section~\ref{section:main} we give the necessary assumptions and
definitions and then state Theorem~\ref{mainlocglob}.   In
Section~\ref{section:example} we give an application of
Theorem~\ref{mainlocglob} to a family of Lipschitz vector fields.
Such vector fields are a natural example of where degenerate
$p$-Laplacians arise.  

A central and somewhat surprising
part of our proof of Theorem~\ref{mainlocglob} is to prove the existence and
boundedness of solutions of the Dirichlet problem
\begin{equation*} 
\begin{cases}
\mx_{p,\tau} u &  = \varphi  \text{ in } \Theta \\
u  & = 0  \text{ in } \partial \Theta,
\end{cases}
\end{equation*}
where $\mx_{p,\tau}$ is a degenerate $p$-Laplacian operator with a
zero order term:
\begin{equation*} 
\mx_{p,\tau} u 
= \text{div}\Big(\big|\sqrt{Q} \nabla u\big|^{p-2}Q\nabla u\Big)
- \tau |u|^{p-2}u.
\end{equation*}
We prove the existence of solutions using Minty's theorem in
Section~\ref{section:weak} and we prove boundedness using ideas
from~\cite{CRW,MRW} in
Section~\ref{section:bdd-solns}.   Finally, in Section~\ref{section:main-proof}
we prove Theorem~\ref{mainlocglob}.  

Throughout this paper, $\Omega$ will be a fixed open, connected subset of
$\R^n$.  We say an open set $\Theta$ is compactly contained in $\Omega$ and
write $\Theta \Subset \Omega$ if $\Theta$ is bounded and
$\bar{\Theta}\subset \Omega$.  The set $Lip_0(\Omega)$ consists of all
Lipschitz functions $f$ such that $\supp(f) \Subset \Omega$.   A
constant $C$ may vary from line to line; if necessary we will denote
the dependence of the constant on various parameters by writing, for
instance, $C(p)$.  

\section{The main result}
\label{section:main}

In order to state Theorem~\ref{mainlocglob} we need to make some
technical assumptions and give some additional definitions.  We begin
with the topological framework.  Fix $\Omega\subset \mathbb{R}^n$ and  let
$\rho:\Omega\times\Omega\ra \mathbb{R}$ be a symmetric quasimetric on
$\Omega$: that is, there is a constant $\kappa\geq 1$ so that for all
$x,y,z\in \Omega$:
\begin{enumerate}
\item $\rho(x,y)=0$ if and only if  $x=y$;

\item $\rho(x,y)=\rho(y,x)$;

\item $\rho(x,y) \leq \kappa\big(\rho(x,z)+\rho(x,y)\big)$.

\end{enumerate}
Given $x\in \Omega$ and $r>0$ we will always denote the $\rho$-ball of
radius $r$ centered at $x$ by $B(x,r)$; that is
$B(x,r)=\{y\in \Omega : \rho(x,y)<r\}$.  We will assume that the balls
$B(x,r)$ are Lebesgue measurable.  We will also use $D(x,r)$
to denote the corresponding Euclidean ball $\{ x\in \Omega : |x-y|<r
\}$.  We will always assume that the quasi-metric $\rho$ and the
Euclidean distance satisfy the following:  
\begin{equation} \label{eqn:cont-metric}
\text{ given }x,\,y\in \Omega,\,
|x-y|\rightarrow 0 \text{ if and only if }\rho(x,y)\rightarrow 0.
\end{equation}
Equivalently, we may assume that given $x\in \Omega$ and any
$\epsilon>0$, there exists $\delta>0$ such that $D(x,\delta)\subset
B(x,\epsilon)$, and that given any $\delta>0$ there exists $\gamma>0$
such that $B(x,\gamma)\subset D(x,\delta)$.   Either of these
conditions hold if we assume that the topology generated by the balls
$B(x,r)$ is equivalent to the Euclidean topology on $\Omega$. 

\begin{rem}
This assumption on the topology of $(\Omega,\rho)$ is taken
from~\cite{MRW}; it is closely related to a condition first assumed by
C.~Fefferman and Phong~\cite{FP} (see also~\cite{SW1}).  
\end{rem}

\medskip

Let $\Sn_{n}$ denote the
collection of all positive, semi-definite $n\times n$ self-adjoint
matrices;   fix a function $Q:\Omega \rightarrow S_{n}$ whose entries
are Lebesgue measurable.  For  a.e.~$x\in
\Omega$ and for all $\xi,\eta\in\mathbb{R}^n$ (where $\xi'$ denotes
the transpose of $\xi$) define the associated quadratic form,
$\qform(x,\xi) = {\lan}Q(x)\xi,\xi{\ran}=\xi'Q(x)\xi$ and inner
product $\lan Q(x)\xi,\eta \ran =\eta'Q(x)\xi$. 
These satisfy
\[ 0\leq \langle Q(x)\xi,\xi \rangle, \qquad
|{\lan}Q(x)\xi,\eta{\ran}|
\leq {\lan}Q(x)\eta,\eta{\ran}^\frac{1}{2}{\lan}Q(x)\xi,\xi{\ran}^\frac{1}{2}.
\]

We define the  operator norm of $Q(x)$ by
\begin{equation*}
\left|Q(x)\right|_\text{op}
=\displaystyle\sup_{|\xi|=1} |Q(x)\xi|.
\end{equation*}
Since $Q$ is semi-definite a.e., $\sqrt{Q}$ is well defined, and
$|\sqrt{Q(x)}|_\text{op} = |Q(x)| _\text{op}^{\frac{1}{2}}$.  We will
write $Q \in L^p(\Omega)$, $1\leq p \leq \infty$, if
$|Q|_\text{op} \in L^p(\Omega)$.

\medskip

Besides the local Sobolev inequality given in
Definition~\ref{localsob}, and which we will hereafter assume holds for
$\rho$-balls $B$, we will also need a local Poincar\'e property.  In
stating it, we assume that  $r_1 : \Omega \rightarrow (0,\infty)$
is the same function that appears in Definition~\ref{localsob}.

\begin{defn}\label{pcprop} 
  Given $1\leq p < \infty$ and $t' \geq 1$, a local Poincar\'e property of order
  $p$ with gain $t'\geq 1$ holds in $\Omega$ if there are
  constants $C_1>0$ and $\beta\geq 1$ such that for any $y\in \Omega$ and
  $r$ such that $0<\beta r<r_1(y)$,
\begin{equation}\label{localpc} 
\bigg(\fint_{B(y,r)}  |f-f_{B(y,r)}|^{pt'}\,dx\bigg)^\frac{1}{pt'}
\leq
  C_1r\bigg(\fint_{B(y,\beta r)}|\sqrt{Q}\nabla
  f|^p\,dx\bigg)^\frac{1}{p}
\end{equation}
holds for every
  $f\in Lip(B(y,\beta r))$ such that $\sqrt{Q}\nabla f\in
  L^p(B(y,\beta r))$, and where $f_{B(y,r)}=\fint_{B(y,r)} f(x)\,dx$.
\end{defn}

\begin{rem} 
  Inequality \eqref{localpc} will be used to establish that the
  embedding of the degenerate Sobolev space $\widehat{H}^{1,p}_{Q,0}$ (that we will
  define below) into  $L^p$ is compact.  We will also use it to prove
  a product rule for functions in this Sobolev space.  The parameter
  $t'$ is determined by the regularity of $Q$:  
  the more regular $Q$ is, the smaller $t'$ is permitted to be.  In
  fact, if $Q$ is locally bounded in $\Omega$, \eqref{localpc} is not
  required to establish the product rule: see the proof of Lemma \ref{prodlem}.
\end{rem} 

\begin{rem}
The problem of determining sufficient conditions on the matrix $Q$ for the
Poincar\'e inequality~\eqref{localpc} to hold has been considered in a
somewhat different form in~~\cite{CIM,CRR}.  It is interesting to note
that in this case the condition involves the solution of a Neumann
problem for a degenerate $p$-Laplacian operator.
\end{rem}


Our final definition is a technical assumption on the geometry of
$(\Omega,\rho)$.  This
condition, which we refer to as the ``cutoff" condition, ensures the
existence of accumulating sequences of Lipschitz cutoff
functions on $\rho$-balls.  Again, the function $r_1$ is assumed to be
the same as in Definition~\eqref{localsob}.  

\begin{defn}\label{aslcof} 
Given $(\Omega,\rho)$ and a matrix $Q$, a cutoff condition of order $1\leq s\leq
  \infty$ holds if there exist constants
  $C_3,\,N >0$ and $0<\alpha<1$ such that given $x\in \Omega$ and $0<r<r_1(x)$ there
 exists a sequence $\{\psi_j\}_{j=1}^\infty\subset Lip_0(B(x,r))$
such that for all $j \in \mathbb{N}$,
\begin{equation}\label{cutoff}
\begin{cases} 0\leq \psi_j \leq 1,\\
\emph{supp } \psi_1 \subset B(x,r),\\
B(x,\alpha r) \subset \{y\in B(x,r) : \psi_j(y)=1\},\\
\emph{supp } \psi_{j+1} \subset \{y\in B(x,r) :\psi_j(y)=1\}, \\
\bigg(\disp\fint_{B(x,r)} |\sqrt{Q(y)}\nabla \psi_j(y)|^s
dy\bigg)^\frac{1}{s} \leq C_3 \frac{N^j}{r}. 
\end{cases}
\end{equation}
\end{defn}

Definition~\ref{aslcof} first appeared in~\cite{SW1}, though it is a
generalization of a concept that has appeared previously in the
literature; see~\cite{SW1} for further references.  If $\rho$ is the
Euclidean metric and $Q$ is bounded, then this sequence of cutoff
functions can be taken to be the standard Lipschitz cutoff functions.  More
generally, it was shown in~\cite{SW1} that with our assumptions on
$\rho$, if $Q$ is continuous, then such a sequence exists with
$s=\infty$.   This cutoff condition holds for a wide variety of
geometries, including $\rho$ which produce highly degenerate balls.
See, for example,~\cite{Mac}.

\begin{rem}\label{kmr}
  There is a close connection between the cutoff condition and
  doubling.  In~\cite{KMR} it was shown that if  the local Sobolev
  property of order $p$ and gain $\sigma$ and the cutoff condition~\eqref{cutoff} both
  hold, then  Lebesgue measure is locally doubling for the collection of
  $\rho$-balls $\{B(x,r)\}_{x\in\Omega; r>0}$.  That is, there exists a
  positive constant $C$ so that given any $x\in\Omega$ and
  $0<r<r_1(x)$ then $\big|B(x,2r)\big| \leq C|B(x,r)|$.  Consequently, for
  any $0<r\leq s<r_1(x)$, 
 \begin{equation} \label{eqn:kmr-doubling}
 \big| B(x,s) \big| \leq \tilde{C} \big(\frac{s}{r}\big)^{d_0}
  \big|B(x,r)|, 
\end{equation}
where $d_0 = \log_2(C)$.  We will use this fact to
  prove Proposition~\ref{CRWpro} below.
\end{rem} 

\medskip

We can now state our main result.

\begin{thm}\label{mainlocglob} 
  Given a set $\Omega \subset \mathbb R^n$, let $\rho$ be a quasi-metric on
$\Omega$.    Fix $1<p<\infty$ and $1<t\leq\infty$, and suppose $Q$ is
a semi-definite matrix function such that 
  $Q\in L^{\frac{pt}{2}}_{\emph{loc}}(\Omega)$.  Suppose further that 
  that the cutoff condition of order $s>p\sigma'$, the local
  Poincar\'e property of order $p$ with gain $t'=\frac{s}{s-p}$, and
  the local Sobolev property of order $p$ with gain 
  $\sigma\geq 1$ hold. Then, given any open
  set $\Theta\Subset \Omega$
  there is a positive constant $C(\Theta)$ such that the global
  Sobolev inequality 
\begin{equation}\label{sobgoal} 
\bigg(\int_\Theta |f|^{p\sigma}dx\bigg)^\frac{1}{p\sigma} 
\leq C(\Theta)\bigg(\int_\Theta
  |\sqrt{Q}\nabla f|^pdx\bigg)^\frac{1}{p} 
\end{equation}
holds for all
  $f\in Lip_0(\Theta)$.
\end{thm}

\begin{rem} 
  If $Q\in L^\infty_{\emph{loc}}(\Omega)$, then we can take $t=\infty$ and
  $t'=1$, so that $s=\infty$.  Thus, we only need to assume a local
  Poincar\'e inequality of order $p$ without gain.  If we assume that
  $\rho$ is the Euclidean metric, then as we noted above, the cutoff
  condition holds with $s=\infty$.  Thus, Theorem~\ref{thm:main-special} follows
  immediately from Theorem~\ref{mainlocglob}.
\end{rem}

\section{Example: diagonal Lipschitz vector fields}
\label{section:example}

In this section we give an illustrative example of the application of
Theorem~\ref{mainlocglob}.  
Let $\Omega$ be any bounded domain in $\mathbb{R}^n$ and let
$1<p<\infty$.   Fix a vector function $a=(a_1,\ldots,a_n)$ where $a_1=1$
and $a_2,\ldots,a_n:\Omega\rightarrow\mathbb{R}$ are such that 
for $2\leq j \leq n$, $a_j$ is bounded, nonnegative and Lipschitz
continuous.  Further, assume that the $a_j$  satisfy the $RH_\infty$
condition in the first variable $x_1$ uniformly in $x_2,\ldots,x_n$:  
there exists a constant $C$ for each interval $I$ and
$x=(x_1,\ldots,x_n)\in \Omega$,
\[  a_j(x_1,\ldots,x_n) \leq C\fint_I a(z_1,x_2,\ldots,x_n)dz_1. \]
For instance, we can take $a_j(x_1,\ldots,x_n)=
|x_1|^{\alpha_j}b_j(x_2,\ldots,x_n)$, where $\alpha_j \geq 0$ and $b_j$
is a non-negative Lipschitz function.  (For more on the $RH_\infty$ condition, see \cite{CN}.)

Now let $X_j=a_j\frac{\partial}{\partial x_j}$ and $\nabla_a =
(X_1,...,X_n)$,  and  define the associated $p$-Laplacian
\begin{equation}\label{alap}
L_{p,a}u=\text{div}_a\left(\left|\nabla_a u\right|^{p-2}\nabla_au\right).
\end{equation}
If we let $Q(x) = \text{diag}(1,a_2^2,...,a_n^2)$, then we have that
$L_{p,a} = \text{div}\big(|\sqrt{Q}\nabla u|^{p-2}Q\nabla
  u\big)$.  It is shown in \cite{MRW2} that each of Definitions
\ref{sobprop}, \ref{pcprop}, and \ref{aslcof} hold with respect the
family of non-interference balls $A(x,r)$ defined as in \cite{SW1},
and there exists a quasi-metric $\rho$ such that the non-interference
balls are equivalent to the $\rho$-balls.  In fact, setting
$B(x,r)=A(x,r)$ and $r_1(x)=\delta'\text{dist}(x,\partial\Omega)$ for
$\delta'>0$ sufficiently small depending on $\|Q\|_{\infty}$,
condition \eqref{cutoff} holds with $s=\infty$, \eqref{localpc} holds
with $t'=1$ and \eqref{localsob} holds with $\sigma=\frac{d_0}{d_0-p}$
where $d_0$ is the doubling exponent associated to Lebesgue measure
and the collection of balls $A(x,r)$, as in \eqref{eqn:kmr-doubling}.
As a result, both \cite[(1.15) and (1.16)]{MRW} hold with $t=\infty$
and $t'=1$.  Therefore, we can apply Theorem~\ref{mainlocglob} to get
that for any open subdomain $\Theta\Subset\Omega$, there exists a constant
$C(\Theta)$ such that the Sobolev inequality
\begin{equation}\label{appsob}
\left(\int_\Theta \left|f\right|^{p\sigma}\,dx\right)^\frac{1}{p\sigma} 
\leq C(\Theta)  \left(\int_\Theta \left|\sqrt{Q}\nabla f\right|^p\,dx\right)^\frac{1}{p}
\end{equation}
holds for all $f\in Lip_0(\Theta)$.  Moreover, by the doubling
property~\eqref{eqn:kmr-doubling}, if we let $\Theta = A(x,r)$ for
$0<r<\delta\text{dist}(x,\partial\Omega)$, then we get
\[
\left(\fint_{A(x,r)}
  \left|f\right|^{p\sigma}\,dx\right)^\frac{1}{p\sigma} 
\leq Cr  \left(\fint_{A(x,r)} \left|\sqrt{Q}\nabla f\right|^p\,dx\right)^\frac{1}{p}
\]
for any $f\in Lip_0(A(x,r))$.

As a consequence, when  $p=2$ inequality \eqref{appsob} is suffiencent to
use \cite[Theorem 3.10]{R} to prove the existence of a unique weak
solution of the linear Dirichlet problem
\begin{equation*}
\begin{cases}
\text{div}\left(Q\nabla u\right) & = f\text{ in }\Theta \\
u &= 0\text{ on }\partial\Theta.
\end{cases}
\end{equation*}

\section{Weak solutions of degenerate $p$-Laplacians}
\label{section:weak}

A key step in the proof of
Theorem~\ref{mainlocglob} is to prove the existence and
boundedness of solutions of the Dirichlet problem
\begin{equation} \label{eqn:dirichlet}
\begin{cases}
\mx_{p,\tau} u &  = \varphi  \text{ in } \Theta \\
u  & = 0  \text{ in } \partial \Theta,
\end{cases}
\end{equation}
where $\mx_{p,\tau}$ is a degenerate $p$-Laplacian operator with a
zero order term:
\begin{equation} \label{eqn:p-laplacian}
\mx_{p,\tau} u 
= \text{div}\Big(\big|\sqrt{Q} \nabla u\big|^{p-2}Q\nabla u\Big)
- \tau |u|^{p-2}u.
\end{equation}
In this section we will define weak solutions to this equation and
prove that they exist.  

As the first step we define the degenerate Sobolev spaces related to
$Q$.  Detailed discussions of these spaces can be found in
\cite{CMN,CRR,CRW,MRW,MRW2,SW2}; here we will sketch the key ideas and
refer the reader to these references for further information.  Fix
$1\leq p<\infty$ and a matrix function $Q$ such that $\sqrt{Q}\in
L^p_{\emph{loc}}(\Omega)$.   Fix an open set $\Theta \Subset \Omega$,
and for  $1\leq p<\infty$ define $\mathcal{L}^p_Q(\Theta)$ to be the collection
of all measurable $\mathbb{R}^n$ valued functions
${\bf f}=(f_1,...,f_n)$ that satisfy
\begin{equation}\label{normLQ}
\|{\bf f}\|_{\mathcal{L}^p_Q(\Theta)} 
=\bigg(\int_\Theta \big|\sqrt{Q}{\bf f}\big|^p\,dx\bigg)^{1/p} <\infty.
\end{equation}
More properly we define $\mathcal{L}^p_Q(\Theta)$ to be the normed vector
space of equivalence classes under the equivalence relation ${\bf f} \equiv {\bf g}$ if
$\|{\bf f}-{\bf g}\|_{\mathcal{L}^p_Q(\Theta)} = 0$.  Note that if
${\bf f}(x)={\bf g}(x)$ a.e., then ${\bf f} \equiv {\bf g}$, but the
converse need not be true, depending on the degeneracy of~$Q$.  

Let $Lip_Q(\Theta)$ be the collection of all functions
$f\in Lip_{loc}(\Theta)$ such that $f\in L^p(\Theta)$ and
$\nabla f\in \mathcal{L}^p_Q(\Theta)$.  We now define the
corresponding degenerate Sobolev space $\widehat{H}^{1,p}_Q(\Theta)$
to be the formal closure of $Lip_Q(\Theta)$ with respect to the norm
\[ \|f\|_{\widehat{H}^{1,p}_Q(\Theta)} = \bigg[\int_\Theta |f|^p\, dx 
+ \int_\Theta \qform(x,\nabla f)^\frac{p}{2}\,dx\bigg]^\frac{1}{p}
=\bigg[\int_\Theta |f|^p\,dx + \int_S |\sqrt{Q}\nabla
f|^p\,dx\bigg]^\frac{1}{p}. \]
Similarly, define $\widehat{H}^{1,p}_{Q,0}(\Theta) \subset
\widehat{H}^{1,p}_Q(\Theta)$ to be the formal 
closure of $Lip_0(\Theta)$  with respect to this norm. 

Because of the degeneracy of $Q$, we cannot represent either
$\widehat{H}^{1,p}_{Q,0}(\Theta)$ or $\widehat{H}^{1,p}_Q(\Theta)$ as
spaces of functions except in special situations.  But, since
$L^p(\Theta)$ and $\mathcal{L}^p_Q(\Theta)$ are complete, given an
equivalence class of $\widehat{H}^{1,p}_Q(\Theta)$ there exists a
unique pair
$\vec{\bf f}=(f,{\bf g})\in L^p(\Theta)\times \mathcal{L}^p_Q(\Theta)$
that we can use to represent it.  Such pairs are unique and so we
refer to elements of $\widehat{H}^{1,p}_Q(\Theta)$ using their
representative pair.  However, because of the classical example in~\cite{FKS},
${\bf g}$ need not be uniquely determined by the first component $f$
of the pair: if we think of ${\bf g}$ as the ``gradient'' of $f$, then
there exist non-constant functions $f$ whose gradient is $0$.  

On the other hand, since $\sqrt{Q}\in L^p_{loc}(\Omega)$ and since
constant sequences are Cauchy, if $f\in Lip_Q(\Theta)$, then
$(f,\nabla f)\in \widehat{H}^{1,p}_Q(\Theta)$ where $\nabla f$ is the
classical gradient of $f$ in $\Theta$:  see~\cite{GT}. 

We need one structural result about these Sobolev spaces.  The
following result is proved in~\cite{CRR} for the space
$H^{1,p}_Q(\Theta)$, which is the closure of $C^1(\bar{\Theta})$ with
respect to the $\widehat{H}^{1,p}_Q(\Theta)$ norm, but the proof is
identical in our case.

\begin{lem} \label{lemma:bfs} 
  Given $1\leq p<\infty$, $\Theta \subset \Omega$, and a matrix $Q$,
  $\widehat{H}^{1,p}_Q(\Theta)$ and $\widehat{H}^{1,p}_{Q,0} (\Theta)$
  are separable Banach spaces.  If $p>1$, they are reflexive.
\end{lem}

\medskip

We now define the weak solution of the Dirichlet
problem~\eqref{eqn:dirichlet} for equation~\eqref{eqn:p-laplacian}.
We will assume that $1<p<\infty$, $\tau\geq 0$,
$\vphi\in L^{p'}_{\text{loc}}(\Omega)$, and 
$\sqrt{Q}\in L^p_{\text{loc}}(\Omega)$.
Associated to the Dirichelt problem is the non-linear form
\[ \ma_{p,\tau}:\widehat{H}_Q^{1,p}(\Theta)\times
  \widehat{H}_Q^{1,p}(\Theta)\ra \mathbb{R}, \]
defined for 
  $\vec{\bf u}=(u,\vec{g})$ and $\vec{\bf v} = (v,\vec{h})$ by
\begin{equation} \label{form} 
\ma_{p,\tau}\vec{\bf u}(\vec{\bf v}) = \int_\Theta
  {\lan}Q\vec{g},\vec{g}{\ran}^\frac{p-2}{2}{\lan}Q\vec{g},\vec{h}{\ran}\,dx
  + \tau\int_\Theta |u|^{p-2}uv\,dx; 
\end{equation}
 we use the convention
  that $\ma_{p,\tau}\vec{\bf 0}(\cdot)=0$  if $1<p<2$.  The notation
  used on the left-hand side of \eqref{form} is meant to suggest that for each fixed
  $\vec{\bf u}=(u,\vec{g})\in \widehat{H}_Q^{1,p}(\Theta)$, the operator
  $\ma_{p,\tau}\vec{\bf u}(\cdot)\in \big(\widehat{H}_Q^{1,p}(\Theta)\big)'$;
  see Lemma \ref{boundlemma} below.

We use this form to define a weak solution.

\begin{defn}\label{weaksoldef} 
A weak solution to the Dirichlet problem~\eqref{eqn:dirichlet} is an
element $\vec{\bf u}=(u,\vec{g})\in \widehat{H}^{1,p}_{Q,0}(\Theta)$
such that the equality
\begin{equation} \label{weaksol}
\ma_{p,\tau}\vec{\bf u}(v) =\int_\Theta
{\lan}Q\vec{g},\vec{g}{\ran}^\frac{p-2}{2}{\lan}Q\vec{g},\nabla
v{\ran}\,dx + \tau\int_\Theta |u|^{p-2}uv\,dx= -\int_\Theta \vphi
v\,dx 
\end{equation}
holds for every $v\in Lip_0(\Theta)$.
\end{defn}

\begin{rem} 
  If $\vec{\bf u}\in \widehat{H}^{1,p}_{Q,0}(\Theta)$ is a weak
  solution, then by a standard density argument we have that
  \eqref{weaksol} holds if we replace $(v,\nabla v)$ with any
  $\vec{\bf v}=(v,\vec{h})\in \widehat{H}^{1,p}_{Q,0}(\Theta)$.
\end{rem}

\medskip

We can now state and prove our existence result.

\begin{pro}\label{existence} 
  Let $\Omega \subset \mathbb R^n$ be open.  Given $1<p<\infty$ and
  $\tau>0$, suppose $Q\in L^\frac{p}{2}_{\emph{loc}}(\Omega)$.  Then,
  for any open set $\Theta\Subset\Omega$ the Dirichlet problem
  \eqref{eqn:dirichlet} with $\vphi\in L^{p'}(\Theta)$ has a weak solution
  $\vec{\bf u}=(u,\vec{g})\in \widehat{H}^{1,p}_{Q,0}(\Theta)$.
\end{pro}

To prove Proposition~\ref{existence} we will use  Minty's theorem as
found in \cite{Sh}; this result is a generalization of the Lax-Milgram
theorem to general Banach spaces.    To state it we fix some
notation.  Let $X$ be a separable, reflexive Banach space with norm
$\|\cdot\|_X$, and let $X^*$  denote its dual space.  Given a map $T :
X \rightarrow X^*$ and $u,\,v \in X$, we will write $T(u)(v)=\langle
T(u),v\rangle$.  

\begin{thm}\emph{(Minty)} \label{thm:minty}
  Let $X$ be a separable, reflexive Banach space and fix $\Gamma \in X^*$.
  Let $T:X\ra X^*$ be an operator that is:
\begin{itemize}
\item  bounded:  $T$ maps bounded subsets of $X$ to bounded subsets of $X'$;

\item monotone:   ${\lan}T(u)-T(v),u-v{\ran} \geq 0$ for all $u,\,v\in
  X$;

\item hemicontinuous:  for $z\in R$, the mapping $z\mapsto T[u+zv](v)$
  is continuous for all $u,\,v\in X$;

\item almost coercive: there exists $\beta>0$ such that
  ${\lan}Tv,v{\ran} > \langle \Gamma, v \rangle$ for all $v\in X$ such that
  $\|v\|_X>\beta$.  

\end{itemize}
Then the set $u \in X$ such that $T(u)=\Gamma$ is non-empty.
\end{thm}

To apply Theorem~\ref{thm:minty} to solve the Dirichlet
problem~\eqref{eqn:dirichlet}, let
$X=\widehat{H}^{1,p}_{Q,0}(\Theta)$; then by Lemma~\ref{lemma:bfs},  $X$ 
is a separable, reflexive Banach space. 
Fix
$\vphi\in L^{p'}(\Theta)$; given $\vec{\bf v}=(v,\vec{h})$, define $\Gamma \in X^*$ by
\begin{equation} \label{eqn:Gamma} 
\Gamma(\vec{\bf{v}}) = - \int_\Theta \varphi v\,dx. 
\end{equation}
Let $T= \ma_{p,\tau}$; then
$\vec{\bf u}\in \widehat{H}^{1,p}_{Q,0}(\Theta)$ is a weak solution if
$\ma_{p,\tau}\vec{\bf u}(\vec{\bf v}) = \Gamma(\vec{\bf v})$ for every
$\vec{\bf v}=(v,\vec{h})\in \widehat{H}^{1,p}_{Q,0}(\Theta)$.  By Minty's
theorem, such a $\vec{\bf u}$ exists if $\ma_{p,\tau}$ is
bounded, monotone, hemicontinuous, and almost coercive.   To complete
the proof of Proposition~\ref{existence}, we will  prove
each of these properties in turn.  

We begin with three useful inequalities which we record as a lemma.
For their proof, see~\cite[Ch. 10]{PL}.

\begin{lem} \label{lemma-PL}
For all $s,\,r \in \mathbb{R}^n$, 
\begin{equation} \label{eqn:pl1}
 \langle |s|^{p-2}s-|r|^{p-2}r, s-r\rangle \geq 0;
\end{equation}
if $p \geq 2$, 
\begin{equation} \label{eqn:pl2}
\big| |s|^{p-2}s-|r|^{p-2}r\big| \leq
c(p)\big(|s|^{p-2}+|r|^{p-2}\big)|s-r|; 
\end{equation}
if $1<p\leq 2$, 
\begin{equation} \label{eqn:pl3}
\big| |s|^{p-2}s-|r|^{p-2}r\big| 
\leq c(p)|s-r|^{p-1}.
\end{equation}
\end{lem}


\begin{lem}\label{boundlemma} $\ma_{p,\tau}$ is bounded on $\widehat{H}^{1,p}_{Q,0}(\Theta)$ for all $1< p<\infty$ and $\tau\in \mathbb{R}$.
\end{lem}

\begin{proof}
  Fix $1< p<\infty$ and $\tau\in \mathbb{R}$.  Let
  $\uu=(u,\vec{g}),\vv=(v,\vec{h})\in
  \widehat{H}^{1,p}_{Q,0}(\Theta)$.  If we apply H\"older's inequality
  twice, then we have that
\begin{align*}
 |\ma_{p,\tau}\uu(\vv)| 
&\leq \bigg|\int_\Theta \lan Q\vec{g},\vec{g} \ran^\frac{p-2}{2} 
\lan Q\vec{g},\vec{h}\ran dx\bigg| 
+ |\tau|\bigg|\int_\Theta |u|^{p-2}uvdx\bigg|\\
  &\leq \int_\Theta \lan Q\vec{g},\vec{g}\ran^\frac{p-1}{2}\lan
    Q\vec{h},\vec{h}\ran^\frac{1}{2}dx +
    |\tau|\|u\|^{p-1}_{L^p(\Theta)}\|v\|_{L^p(\Theta)}\\ 
  &\leq \|\sqrt{Q}\;
    \vec{g}\|^{p-1}_{L^p(\Theta)}\|\sqrt{Q}\;\vec{h}\|_{L^p(\Theta)}
+|\tau|\|u\|^{p-1}_{L^p(\Theta)}\|v\|_{L^p(\Theta)} \\ 
  &\leq
  (1+|\tau|)\|\uu\|_{\widehat{H}^{1,p}_{Q,0}(\Theta)}^{p-1}
\|\vv\|_{\widehat{H}^{1,p}_{Q,0}(\Theta)}.
\end{align*}
It follows at once from this inequality that $\ma_{p,\tau}$ is bounded.
\end{proof}


\begin{lem}\label{monolemma} 
$\ma_{p,\tau}$ is monotone for all $1< p<\infty$ and $\tau\geq 0$.
\end{lem}

\begin{proof}
 Fix $p$ and $\tau$, and let $\uu,\vv\in \widehat{H}^{1,p}(\Theta)$ be as in
 Lemma~\ref{boundlemma}.    Then,
\begin{align*}
\langle \ma_{p, \tau}\uu-\ma_{p, \tau}\vv, \uu-\vv \rangle
& = \ma_{p, \tau}\uu(\uu-\vv) - \ma_{p, \tau} \vv(\uu-\vv) \\
&=\int_{\Omega} \langle Q\vec{g},\vec{g} \rangle ^\frac{p-2}{2}
  \langle Q\vec{g}, 
\vec{g} - \vec{h} \rangle \,dx \\
& \qquad \qquad  -\;\int_{\Omega} \langle Q\vec{h}, \vec{h} \rangle
  ^\frac{p-2}{2} \langle Q\nabla \vec{h}, \vec{g} - \vec{h} \rangle
  \,dx   \\
& \qquad \qquad  +\;\tau\bigg(  \int_{\Omega}  (|u|^{p-2} u-|v|^{p-2}v)
  (u- v) \,dx \bigg)\\ 
&= I_1+\tau I_2.
\end{align*}

We estimate $I_1$ and $I_2$  separately, beginning with $I_2$.   By
inequality~\eqref{eqn:pl1}, 
\[  I_2 = \int_{\Omega}  \langle |u|^{p-2}u-|v|^{p-2}v, u - v
  \rangle \,dx\geq 0. \]

To estimate $I_1$ note that by the symmetry of $Q$ we have that
$\lan Q\vec{g},\vec{g}\ran^\frac{1}{2} = |\sqrt{Q}\vec{g}|$.  Hence,
\begin{align*}
I_1
&= \int_\Theta |\sqrt{Q}\vec{g}|^{p-2}\lan
  Q\vec{g},\vec{g}-\vec{h}\ran \,dx 
- \int_\Theta |\sqrt{Q}\vec{h}|^{p-2}\lan Q\vec{h}, \vec{g}-\vec{h}\ran \,dx\\
&= \int_\Theta \lan |\sqrt{Q}\vec{g}|^{p-2}\sqrt{Q}\vec{g} -
|\sqrt{Q}\vec{h}|^{p-2}\sqrt{Q}\vec{h}, \sqrt{Q}\vec{g} -
\sqrt{Q}\vec{h}\ran \,dx.  
\end{align*}
With $s=\sqrt{Q}\,\vec{g}$ and
$r=\sqrt{Q}\,\vec{h}$, the integrand is again of the form
\eqref{eqn:pl1} and so non-negative.  Thus
$I_1\geq 0$ and our proof is complete.
\end{proof}


\begin{lem}\label{hemilemma} $\ma_{p,\tau}$ is hemicontinuous for all $1<p<\infty$ and $\tau\in \mathbb{R}$.
\end{lem}

\begin{proof}
Fix $\uu,\vv\in \widehat{H}^{1,p}_{Q,0}(\Theta)$ as in the
previous lemmas.  For $z\in \mathbb{R}$, let
$z\vv = z(v,\vec{h})=(zv,z\vec{h})\in \widehat{H}^{1,p}_{Q,0}(\Theta)$; we will
show that the function $z \mapsto \ma_{p, \tau}(\uu+z \vv)(\vv)$ is
continuous. By the definition of
$\ma_{p,\tau}$ we can split this mapping into the sum of two parts:
\begin{gather*}
z \mapsto \mathcal{G}_{p}(\uu+z \vv)(\vv) 
= \int_{\Theta} \langle Q (\vec{g} + z\vec{h}), 
(\vec{g} + z\vec{h}) \rangle ^\frac{p-2}{2} \langle Q (\vec{g} +
  z\vec{h}), 
\vec{h} \rangle \mathrm{d}x, \\ 
z \mapsto \mathcal{H}_{p, \tau}(\uu+z \vv)(\vv)  
= \int_{\Theta} \tau |u+ z v|^{p-2} (u+z v) v \mathrm{d}x. 
\end{gather*}
We will show each part is continuous in turn.

To show that the mapping $z\mapsto \mathcal{G}_{p}(\uu+z\vv)(\vv)$ is
continuous, we modify an argument from the proof of
\cite[Proposition~3.15]{CMN}.  Fix $z,\,w \in \mathbb{R}$; then (since
$Q=\sqrt{Q}\sqrt{Q}$ is symmetric),
\begin{align*}
& |\mathcal{G}_p(\uu+z\vv)(\vv)-\mathcal{G}_p(\uu+w\vv)(\vv)| \\
& \quad =
\bigg| \int_\Theta |\sqrt{Q}(\vec{g}+z\vec{h})|^{p-2}
\big\langle \sqrt{Q}(\vec{g}+z\vec{h}),
  \sqrt{Q}\vec{h}\big\rangle\,dx \\
& \quad \qquad \qquad -
\int_\Theta |\sqrt{Q}(\vec{g}+w\vec{h})|^{p-2}
\big\langle \sqrt{Q}(\vec{g}+w\vec{h}), 
  \sqrt{Q}\vec{h}\big\rangle\,dx \bigg|\\
& \quad \leq 
\int_\Theta \big| |\sqrt{Q}(\vec{g}+z\vec{h})|^{p-2}
\sqrt{Q}(\vec{g}+z\vec{h}) 
- |\sqrt{Q}(\vec{g}+w\vec{h}|^{p-2}
\sqrt{Q}(\vec{g}+w\vec{h})\big|\;|\sqrt{Q}\vec{h}|\,dx;
\intertext{if $p\geq 2$, then by \eqref{eqn:pl2} and H\"older's
  inequality with exponent $\frac{p}{p-2}$,}
& \quad \leq C(p)
\int_\Theta \big( |\sqrt{Q}(\vec{g}+z\vec{h}) |^{p-2}
+ |\sqrt{Q}(\vec{g}+w\vec{h})|^{p-2}\big)
|z-w|
|\sqrt{Q}\vec{h}|^2\,dx \\
& \quad \leq C(p)
\bigg(\int_\Theta \big( |\sqrt{Q}(\vec{g}+z\vec{h}) |^{p-2}
+
  |\sqrt{Q}(\vec{g}+w\vec{h})|^{p-2}\big)^{\frac{p}{p-2}}\,dx\bigg)^{\frac{p-2}{p}} \\
& \quad \qquad \qquad \times |z-w|
\bigg(\int_\Theta |\sqrt{Q}\vec{h}|^p\,dx\bigg)^{\frac{2}{p}} \\
& \quad \leq C(p) |z-w|
\big( \|\vec{g}\|_{\mathcal{L}^p_Q(\Theta)}
+(|z|+|w|) \|\vec{h}\|_{\mathcal{L}^p_Q(\Theta)}\big)^{p-2}
  \|\vec{h}\|_{\mathcal{L}^p_Q(\Theta)}^2. 
\end{align*}
Since the norms in the final term are all finite, we see that this
term tends to $0$ as $w\rightarrow z$; thus the mapping
$z\mapsto \mathcal{G}_{p}(\uu+z\vv)(\vv)$ is continuous. when $p\geq
2$.

When $1<p<2$, we can essentially repeat the above argument but instead
apply~\eqref{eqn:pl3} to get that 
\[
|\mathcal{G}_p(\uu+z\vv)(\vv)-\mathcal{G}_p(\uu+w\vv)(\vv)| 
\leq C(p) \int_\Theta |z-w|^{p-1}|\sqrt{Q}\vec{h}|^p\,dx, 
\]
and the desired continuity again follows.
\medskip

To show that the mapping $z \mapsto \mathcal{H}_{p, \tau}(\uu+z
\vv)(\vv) $ is continuous,  again fix $z,\,w \in \mathbb{R}$.  Then
\begin{align*}
\big|\mathcal{H}_{p, \tau}(\uu+z \vv)(\vv)
& -\mathcal{H}_{p, \tau}(\uu+w \vv)(\vv)\big|\\
&= \left| \int_{\Theta} \tau |u+ z v|^{p-2}(u+ z v) 
-\tau |u + w v|^{p-2} (u+ wv) v\,dx\right|\\
&\leq  |\tau|\int_{\Theta}  \left| |u+ z v|^{p-2}(u+ z v)  
-|u + w v|^{p-2} (u+ w v) \right |\,|v|\,dx.
\end{align*}
The integrand in the final term tends to $0$ pointwise as
$w\rightarrow z$, so the desired continuity will follow by the
dominated convergence theorem if we can prove that the integrand is
dominated by an integrable function.   But we have that
\begin{align*}
& \left| |u+ z v|^{p-2}(u+ z v)  
-|u + w v|^{p-2} (u+ w v) \right |\,|v| \\
 &\qquad \qquad \leq |u+ z v|^{p-1}|v| +|u + w v|^{p-1}|v| \\
&\qquad \qquad \leq C(p) (|u|^{p-1}|v| +(|z|^{p-1} + |w|^{p-1})|v|^p)\\
& \qquad \qquad \leq C(p,|z|)(|u|^{p-1}|v| + |v|^p).
\end{align*}
By H\"older's inequality, the final term is in $L^1(\Theta)$.  Hence,
we have that the mapping $z\mapsto \mathcal{H}_{p,\tau}(\uu+z\vv)(\vv)$ is
continuous and this completes the proof.
\end{proof}


\begin{lem}\label{complemma} 
Given  $1<p<\infty$ and $\vphi\in L^{p'}(\Theta)$,  define
$\Gamma$ by \eqref{eqn:Gamma}.  Then for all $\tau>0$,  $\ma_{p,\tau}$
is  almost coercive.
\end{lem}

\begin{proof}
  Let $\uu=(u,\vec{g})\in \widehat{H}^1_{Q,0}(\Theta)$ and $\vphi\in
  L^{p'}(\Theta)$.  Then we have that
\begin{multline*}
\ma_{p, \tau}\uu(\uu) 
= \int_{\Theta} \langle Q \vec{g}, \vec{g} \rangle ^\frac{p-2}{2}
\langle Q\vec{g}, \vec{g} \rangle \,dx + \tau \int_{\Theta}
|u|^{p-2} u^2 \,dx  \\
= \int_{\Theta} \langle Q \vec{g}, \vec{g} \rangle
^\frac{p}{2}\,dx 
+ \tau \int_{\Theta}  |u|^p \,dx
\geq\eta\bigg( \int_{\Theta} \langle Q \vec{g}, \vec{g} \rangle ^\frac{p}{2}\,dx +  \int_{\Theta}  |u|^p \,dx\bigg)
= \eta \|\uu\|_{\widehat{H}^{1,p}_{Q,0}}^p,
\end{multline*}
where $\eta = \min\{1,\tau\}>0$. 
Therefore, if we let $\beta=
\big(\eta^{-1}\|\vphi\|_{p'}\big)^{p'-1}$, then by H\"older's inequality we have that
for all $\|\uu\|_{\widehat{H}^{1,p}_{Q,0}(\Theta)}>\beta$,
\[ 
|\Gamma(v)| 
= \bigg|\int_{\Theta} \vphi u \,dx\bigg| 
\leq \|\vphi\|_{p'} \|u\|_p
\leq \|\vphi\|_{p'} \|\uu\|_{\widehat{H}^{1,p}_{Q,0}(\Theta)}
< \eta \|\uu\|_{\widehat{H}^{1,p}_{Q,0}(\Theta)}^p \leq 
 \ma_{p, \tau}\uu(\uu). \]
Thus, $\ma_{p,\tau}$
is  almost coercive.
\end{proof}


\section{Boundedness of solutions to degenerate 
$p$-Laplacians} 
\label{section:bdd-solns}

In this section we will prove that solutions of the Dirichlet
problem~\eqref{eqn:dirichlet} are bounded.  The proof is quite
technical, as it relies on a very general result from~\cite{MRW} and
much of the work in the proof is checking the hypotheses.

\begin{pro} \label{tauindep} 
  Given a set $\Omega \subset \R^n$, let $\rho$ be a quasi-metric on
$\Omega$.    Fix $1<p<\infty$ and $1<t\leq\infty$, and suppose $Q$ is
a semi-definite matrix function such that 
  $Q\in L^{\frac{p}{2}}_{\emph{loc}}(\Omega)$.  Suppose further that 
  that the cutoff condition of order $s>p\sigma'$, the local
  Poincar\'e property of order $p$ with gain $t'=\frac{s}{s-p}$, and
  the local Sobolev property of order $p$ with gain 
  $\sigma\geq 1$ hold. 
Given any  $\Theta\Subset\Omega$ and $q\in
[p',\infty)\cap(p\sigma',\infty)$, if $\vphi\in L^q(\Theta)$, 
  then there exists a positive
  constant $C$ such that for all $\tau\in (0,1)$, the corresponding
  weak solution
  $\vec{\bf u}_\tau=(u_\tau,\vec{g}_\tau)\in \widehat{H}^1_{Q,0}(\Theta)$ of
  \eqref{eqn:dirichlet} satisfies
\begin{equation} \label{eqn:tauindep1}
\esup_{x\in \Theta} |u_\tau(x)| \leq
  C\|\vphi\|_{L^q(\Theta)}^\frac{1}{p-1}. 
\end{equation}
The constant  $C$ is independent
  of $\vphi$, $\vec{\bf u}_\tau$, and $\tau$.
\end{pro}    

\begin{rem}
The hypotheses of Proposition~\ref{tauindep} are the same as those of
Theorem~\ref{mainlocglob} except that we do not require higher
integrability on $Q$: $Q\in L^{\frac{p}{2}}_{\emph{loc}}(\Omega)$ is
sufficient for this result.
\end{rem}

The proof of Proposition~\ref{tauindep} requires that the mapping
$I:\widehat{H}^{1,p}_{Q,0}(\Theta) \ra L^p(\Theta)$,
$I((u,\vec{g})) = u$, is compact.   This is a consequence of the
following result.  

\begin{pro}\label{CRWpro}
Given a set $\Omega \subset \R^n$, let $\rho$ be a quasi-metric on
$\Omega$.    Fix $1<p<\infty$ and $1<t\leq\infty$, and suppose $Q$ is
a semi-definite matrix function such that 
  $Q\in L^{\frac{p}{2}}_{\emph{loc}}(\Omega)$.  Suppose further that 
  that the cutoff condition of order $s>p\sigma'$, the local
  Poincar\'e property of order $p$ with gain $t'=\frac{s}{s-p}$, and
  the local Sobolev property of order $p$ with gain 
  $\sigma\geq 1$ hold. 
  Fix $\Theta\Subset\Omega$; then the mapping
  $I:\widehat{H}^{1,p}_{Q,0}(\Theta) \ra L^p(\Theta)$,
  $I((u,\vec{g}))=u$, is compact.
\end{pro}

\begin{proof}
  Proposition \ref{CRWpro} is a particular case of a general
  imbedding result from~\cite[Theorem~3.20]{CRW}.    So to prove it we
  only need to show that the hypotheses of this result are satisfied.
  We will go through these in turn but for brevity we have omitted
  restating the precise form of each hypothesis as given there and
  instead refer to them as they are stated in the theorem
  and the preliminaries in~\cite[Section~3]{CRW}.  We
  refer the reader to this paper for complete details. 

Since $(\Omega,\rho)$ is a quasi-metric space, $\rho$ satisfies
\eqref{eqn:cont-metric}, and by Remark~\ref{kmr} Lebesgue measure
satsfies a local doubling property for metric balls, the topological
assumptions of Section~3A and condition (3-12) in~\cite{CRW} hold.  In
particular, since we assume that the function $r_1$ in
Remark~\ref{kmr} is continuous, the local geometric doubling condition
in~\cite[Definition~3.3]{CRW} holds.  

In the definition of the underlying Sobolev spaces, and in the
Poincar\'e and Sobolev inequalities~\cite[Definitions~3.5, 3.16]{CRW}, we let the
measures $\mu,\,\nu,\,\omega$ all be the Lebesgue measure.  Since the
local Poincar\'e inequality of order $p$, Defintion \ref{pcprop},
holds, and since $r_1$ is assumed to be continuous, \cite[Definition
3.5]{CRW} holds.  (See also \cite[Remark 3.6]{CRW}.)  Similarly, since
we assume that the local Sobolev inequality of order $p$ with gain
$\sigma$, Definition~\eqref{localsob}, holds, \cite[Definition 3.16]{CRW} holds.

The existence of an accumulating sequence of cut-off functions,
Definition~\ref{aslcof}, lets us prove the cut-off property of order
$s\geq p\sigma'$ in \cite[Definition 3.18]{CRW}.  Fix a compact subset $K$ of $\Omega$
and let $R>0$ be the minimum of $r_1$ on $K$.  Given $x\in K$
and $0<r< R$,   since $r<r_1(x)$, the cutoff condition of order
$s\geq p\sigma'$ gives $\psi_1\in Lip_0(B(x,r))$ such that
\begin{enumerate}
\item $0\leq \psi_1(x)\leq 1$,
\item $\psi_1(x)=1$ on $B(y,\alpha r)$, 
\item $\nabla \psi_1 \in L^s_Q(\Omega)$.
\end{enumerate}
This yields the desired function in~\cite[Definition 3.18]{CRW}.  

Finally, we show that the weak Sobolev inequality, \cite[ Inequality
(3.33)]{CRW}, holds with $t'=\frac{s}{s-p}$.  Since
$s\geq p\sigma'$, $t=\frac{s}{p} \geq \sigma'$,  and so
$1<t'\leq\sigma$.  Again let $K$ be  a compact subset of $\Omega$
and $R>0$ the minimum value of $r_1$ on $K$.  If
$x\in K$ and $0<r<R$, then by the local Sobolev property with
$B=B(x,r)$ we have that for any $u\in Lip_0(B)$,
\[ \bigg(\int_{B}|u|^{pt'}dy\bigg)^\frac{1}{pt'} \leq
  C(B)\bigg(\int_{B}|u|^{p\sigma}dy\bigg)^\frac{1}{p\sigma} \leq
  C(B)\|(u,\nabla u)\|_{\widehat{H}^{1,p}_{Q,0}(\Omega)}, \]
which gives us inequality (3.33).

Thus, we have shown that we satisfy the necessary hypotheses, and so
Proposition~\ref{CRWpro} follows from~\cite[Theorem 3.20]{CRW}.
\end{proof}

\medskip

\begin{proof}[Proof of Proposition~\ref{tauindep}]
  Let $\Theta\Subset\Omega$ and
  $q\in [p',\infty)\cap(p\sigma',\infty)$ with $\vphi\in L^q(\Theta)$.
  Note that since $q\geq p'$ and $\Theta$ is bounded, it follows that
  $\vphi\in L^{p'}(\Theta)$.  Fix $0<\tau<1$; then by
  Proposition~\ref{existence} there exists a weak solution
  $\uu_\tau = (u_\tau,\vec{g}_\tau)\in
  \widehat{H}^{1,p}_{Q,0}(\Theta)$ of the Dirichlet
  problem~\eqref{eqn:dirichlet}.  To complete the proof, we will first use~\cite[Theorem~1.2]{MRW}
to show that $\uu_\tau$ satisfies~\eqref{eqn:tauindep1}.  Then we will
show using Proposition~\ref{CRWpro} that the constant is independent
of $\tau$.  (By \cite[Theorem~1.2]{MRW} we have that it is independent
of $\varphi$ and $\uu_\tau$.)

To apply~\cite[Theorem~1.2]{MRW}, first note that $(\Omega,\rho)$ is a
quasi-metric space and~\eqref{eqn:cont-metric} holds, we satisfy the
topological assumptions of this paper, including~\cite[(1.9)]{MRW}. As a result, if
$0< \beta r<r_1(y)$, the local Sobolev condition,
Definition~\ref{sobprop}, the Poincar\'e inequality,
Definition~\ref{pcprop}, and the cutoff condition,
Definition~\ref{aslcof}, hold. This shows that assumptions~\cite[(1.13), (1.14), (1.15),
(1.16)]{MRW} hold with $t'= \frac{s}{s-p}$.  (For condition (1.16),
see also \cite[Remark~1.1]{MRW}.)

We now show that $\uu_\tau$ is the
solution of an equation with the appropriate properties.   Define 
  $A,\tilde{A}:\Theta\times \mathbb{R}\times\mathbb{R}^n \ra
  \mathbb{R}^n$, and
  $B:\Theta\times \mathbb{R}\times\mathbb{R}^n\ra \mathbb{R}$ by
\begin{enumerate}

\item $A(x,z,\xi) = {\lan}Q(x)\xi,\xi{\ran}^\frac{p-2}{2}Q(x)\;\xi$,

\item $\tilde{A}(x,z,\xi)={\lan}Q(x)\xi,\xi{\ran}^\frac{p-2}{2}\sqrt{Q(x)}\;\xi $,

\item $B(x,z,\xi)=\vphi(x) + \tau|z|^{p-2}z \label{struct3}$.
\end{enumerate}
Given these functions, the differential
equation~\eqref{eqn:p-laplacian} can be rewritten as
\[   \text{div}\big(A(x,u,\nabla u)\big) = B(x,u,\nabla u),  \]
which is \cite[(1.1)]{MRW}.
  Furthermore, we have that  $A(x,z,\xi) = \sqrt{Q(x)}\, \tilde{A}(x,z,\xi)$, and for
  all $(z,\xi)\in \mathbb{R}\times\mathbb{R}^n$ and a.e.~$x\in\Theta$,
\begin{enumerate}
\item $\xi\cdot A(x,z,\xi)  =
  \;{\lan}Q(x)\;\xi,\xi{\ran}^\frac{p-2}{2}{\lan}Q(x)\;\xi,\xi{\ran} = |\sqrt{Q}\;\xi|^p$,

\item  $|\tilde{A}(x,z,\xi)|  = |\sqrt{Q(x)}\;\xi|^{p-1}$,

\item $|B(x,z,\xi)| \leq |\vphi(x)| + \tau|z|^{p-1}<  |\vphi(x)| + |z|^{p-1}$.

\end{enumerate}
Therefore, the structural conditions 
\cite[(1.3)]{MRW} are satisfied with
 the exponents $\delta=\gamma=\psi=p$ and the coefficients
$a=1,\;b=0,\;c=0,\;d=1,\;e=0,\;f=|\vphi|, \;g=0$, and $h=0$.  

The above shows that we satisfy the hypotheses
of~\cite[Theorem 1.2]{MRW}.  Therefore,   fix $\epsilon\in(0,1]$
such that $p-\epsilon >1$ and for $k>0$ (to be fixed below) set $\overline{u}_\tau =
|u_\tau|+k$.   Then for each $y\in\Theta$ and $0<\beta r< r_1(y)$, we have the $L^\infty$-estimate
\begin{equation}\label{localbound}
\esup_{x\in B(y,\alpha r)}|\overline{u}_\tau(x)| 
\leq C Z\bigg[ \frac{1}{|B(y,r)|}\int_{B(y,r)} |\overline{u}_\tau|^{s^*p}\,dx\bigg]^\frac{1}{s^*p},
\end{equation}
where $s^*$ is the dual exponent of $s_1 >\sigma'$, define by
$s=s_1p$.  (In~\eqref{localbound}, $\alpha$ is the constant
from~\eqref{cutoff}; note that in \cite{MRW} it is called $\tau$.)
The term $Z$  on the right-hand side is defined by
\[ Z =
  \left[1+\left(r^p|B(y,r)|^{-\frac{p-\epsilon}{p\sigma'}}
\|1+k^{1-p}|\varphi|\|_{L^\frac{p\sigma'}{p-\epsilon}(B(y,r))}\right)
^\frac{1}{\epsilon}\right]^\frac{s^*}{\sigma-s^*}. 
\]
Since $\frac{p\sigma'}{p-\epsilon}<q$ and $\varphi \in L^q(\Theta)$,
$Z$ is bounded with a bound independent of $\tau$ for any
$k>0$ but depending on the ball $B(y,r)$.

If
$\varphi\neq 0$, let
$k = \|\varphi\|_{L^q(\Theta)}^\frac{1}{p-1}>0$; then by  Minkowski's
inequality and the local Sobolev inequality~\eqref{localsob}, ~\eqref{localbound} becomes
\begin{align*}
\esup_{x\in B(y,\alpha r)}|u_\tau(x)| 
& \leq C\left[\left(\frac{1}{|B(y,r)|}\int_{B(y,r)}
    |u_\tau|^{p\sigma}\;dx\right)^\frac{1}{p\sigma} 
+ \|\varphi\|_{L^q(\Theta)}^\frac{1}{p-1}\right] \\
& \leq
C\left[\frac{1}{|B(y,r)|^\frac{1}{p}}\|\uu_\tau\|_{\widehat{H}^{1,p}_{Q,0}(\Theta)} 
+ \|\varphi\|_{L^q(\Theta)}^\frac{1}{p-1}\right].
\end{align*}
The constant $C$ depends on $B(y,r)$ but not on $\tau$ or $\varphi$.
The case when $\varphi = 0$ is similar and left to the reader.

We now extend this estimate to all of $\Theta$ using the fact that
$\overline{\Theta}$ is compact.  By assumption~\eqref{eqn:cont-metric}, the balls
$B(y,r)$ are open, so we can find a finite cover $B_j =
\tilde{B}(y_j,r_j)$, $1\leq j \leq N$, with  $0<\beta r_j<r_1(y_j)$.  Hence, we have that
\begin{equation}\label{globalbound}
\disp\esup_{y\in \Theta} |u_\tau(x)| 
\leq C_0\left[\|\uu_\tau\|_{\widehat{H}^{1,p}_{Q,0}(\Theta)} +
  \|\varphi\|_{L^q(\Theta)}^\frac{1}{p-1}\right], 
\end{equation}
where the constant $C_0$ depends on $\min\{ r_j : 1\leq j \leq N \}>0$
but not on $\tau$ or $\varphi$.

\medskip

To complete the proof we will show that 
\begin{equation}\label{tauindepgoal} 
\|\uu_\tau\|_{\widehat{H}^{1,p}_{Q,0}(\Theta)}\leq C\|\vphi\|_{L^q(\Theta)}^\frac{1}{p-1}
\end{equation}
with a constant $C$ independent of $\tau$.  To do so we will use
Proposition~\ref{CRWpro}.  Suppose to the contrary
that~\eqref{tauindepgoal} is false.  Then there exists a sequence
$\{\tau_k\}\subset (0,1)$ and corresponding sequence of weak solutions
$\{\uu_{\tau_k} \}=\{u_{\tau_k},\vec{g}_{\tau_k}\} \subset
\widehat{H}^{1,p}_{Q,0}(\Theta)$ of~\eqref{eqn:dirichlet} such that 
\[ \|\uu_\tau \|_{\widehat{H}^{1,p}(\Theta)} \ra \infty  \]
as $k\ra\infty$.   We must have that
$ \tau_k \rightarrow 0$ as $k\rightarrow \infty$.  To see this, note that since
$\uu_{\tau_k}$ is a valid test function in the definition of a weak
solution, we have that
\begin{multline*}
 \tau_k \|\uu_{\tau_k}\|_{\widehat{H}^{1,p}(\Theta)}^p 
= \tau_k\bigg[ \int_\Theta |u_{\tau_k}|^{p-2}u_{\tau_k} u_{\tau_k}\,dx 
+ \int_\Theta \lan
  Q\vec{g}_{\tau_k},\vec{g}_{\tau_k}\ran^\frac{p-2}{2}\lan
  Q\vec{g}_{\tau_k},\vec{g}_{\tau_k}\ran \,dx\bigg] \\
\leq \bigg|\ma_{p,\tau_k}\uu_{\tau_k} (\uu_{\tau_k})\bigg|
= \bigg|\int_\Theta u_{\tau_k} \vphi \,dx\bigg| \leq
\|\uu_{\tau_k}\|_{\widehat{H}^{1,p}(\Theta)}
\|\vphi\|_{L^{p'}(\Theta)}.
\end{multline*}
Since 
$\|\uu_{\tau_k}\|_{\widehat{H}^{1,p}(\Theta)}\neq 0$,this inequality
implies
that 
\[  \|\uu_{\tau_k}\|_{\widehat{H}^{1,p}(\Theta)} \leq
\bigg(\frac{1}{\tau_k}\|\vphi\|_{L^{p'}(\Theta)} \bigg)^\frac{1}{p-1}
\]
which in turn implies that $\tau_k\rightarrow 0$.

For each $k\in \mathbb{N}$  define $\vv_{\tau_k}
=\|\uu_{\tau_k}\|_{\widehat{H}^{1,p}_{Q,0}(\Theta)}^{-1}\uu_{\tau_k}$.
Then, $\vv_{\tau_k}=(v_{\tau_k},\vec{h}_{\tau_k})\in
\widehat{H}^{1,p}_{Q,0}(\Theta)$,
$\|\vv_{\tau_k}\|_{\widehat{H}^{1,p}_{Q,0}(\Theta)} =1$, and
$\vv_{\tau_k}$ is a weak solution of the Dirichlet problem
\[ 
\begin{cases}
\text{div}\big({\lan}Q\nabla w,\nabla
w{\ran}^\frac{p-2}{2}Q\nabla w\big) - \tau_k |w|^{p-2}w 
&= \vphi_k\text{ in } \Theta\\
w& =0\text{ on } \partial\Theta
\end{cases}
\]
where
$\vphi_k =
\|\uu_{\tau_k}\|_{\widehat{H}^{1,p}_{Q,0}(\Theta)}^{1-p}\varphi$.  By
Proposition \ref{CRWpro}, $\widehat{H}^{1,p}_{Q,0}(\Theta)$ is
compactly embedded in $L^p(\Theta)$.  Therefore, since
$\{\vv_{\tau_k}\}$ is a bounded sequence in
$\widehat{H}^{1,p}_{Q,0}(\Theta)$,  by passing to a subsequence
(renumbered for simplicity of notation) we have that there exists
$v\in L^p(\Theta)$ such that $v_{\tau_k}\ra v$ in $L^p(\Theta)$.  
Furthermore, arguing as we did above to prove $\tau_k\rightarrow 0$,
we have a Caccioppoli-type estimate: 
\begin{multline*}
 \int_\Theta
 {\lan}Q\vec{h}_{\tau_k},\vec{h}_{\tau_k}{\ran}^\frac{p}{2}\,dx 
=  \int_\Theta
{\lan}Q\vec{h}_{\tau_k},\vec{h}_{\tau_k}{\ran}^\frac{p-2}{2}
{\lan}Q\vec{h_{\tau_k}},\vec{h}_{\tau_k}{\ran}\,dx 
\\
\leq \ma_{p,\tau_k}\vv_{\tau_k}(\vv_{\tau_k})
\leq \bigg|\int_\Theta v_{\tau_k}\vphi_k\,dx\bigg| 
\leq \|\vphi_k\|_{L^{p'}(\Theta)}.  
\end{multline*}
By definition, 
$\|\vphi_k\|_{L^{p'}(\Theta)}\ra 0$ as $k\ra\infty$; hence, 
$\vv_{\tau_k}\ra \vv_0=(v,0)$ in $\widehat{H}^{1,p}_{Q,0}(\Theta)$ norm.
But then we have that
\[ \|v\|_{L^p(\Theta)} =
  \|\vv_0\|_{\widehat{H}^{1,p}_{Q,0}(\Theta)}=1. \]

By the Poincar\'e inequality \eqref{localpc} we have that for  any
$y\in \Theta$ and $r>0$ sufficiently small, 
\[ \int_{B(y,r)} |v-v_{B(y,r)}|^pdx = 0.  \]
Hence $v$ is constant a.e. (or, more properly, constant on each
connected component of $\Theta$).  We claim that $v=0$ a.e., which
would contradict the fact that $\|v\|_{L^p(\Theta)}=1$.  To show this,
extend ${\vv_0}$ to all of $\Omega$ as follows.  Let
$\{\psi_j\}\subset Lip_0(\Theta)$ be such that
$(\psi_j,\nabla\psi_j)\rightarrow {\vv_0}$ in
$\widehat{H}^{1,p}_{Q,0}(\Theta)$ norm.  Define $\eta_j \in Lip_0(\Omega)$
so that $\eta_j=\psi_j$ in $\Theta$ and $\eta_j=0$ on $\Omega\setminus
\Theta$.  Then $\{(\eta_j,\nabla\eta_j)\}$
is Cauchy in the $\widehat{H}^{1,p}_{Q}(\Omega)$ norm and so
converges to some 
$\ww_0=({w},0)\in
\widehat{H}^{1,p}_{Q,0}(\Omega)$.  If we again
apply 
Poincar\'e's inequality,  we see that $\bar{v}$ is constant in $\Omega$
(since $\Omega$ is connected).   However, $w = v$ in
$\Theta$ and $w=0$ in $\Omega\setminus\Theta$, and so we must have
that $v=0$ a.e.   

From this contradiction we have that our assumption is false and
so~\eqref{tauindepgoal} holds with a constant independent of $\tau$.
This completes our proof.
\end{proof}

\section{Proof of Theorem \ref{mainlocglob}}
\label{section:main-proof}

Before proving our main result, we give one more lemma, a product rule
for degenerate Sobolev spaces.  The proof is adapted from the proof
of~\cite[Proposition~2.2]{MRW}.

\begin{lem} \label{prodlem} 
Given $1\leq p<\infty$ and $1<t\leq\infty$, 
suppose $\sqrt{Q}\in L^{pt}_{loc}(\Omega)$ and
  that the local Poincar\'e property of order $p$ with gain
  $t'$ holds.    Let $\Theta\Subset\Omega$.  If
  $\uu=(u,\vec{g})\in \widehat{H}^{1,p}_Q(\Theta)$ and
  $v\in Lip_0(\Theta)$, then
  $(uv,v\vec{g} + u\nabla v)\in \widehat{H}^{1,p}_{Q,0}(\Theta)$.
\end{lem}

\begin{proof}
  By assumption, there exists a sequence $\{\psi_j\}\subset Lip_Q(\Theta)$ such that 
 $\psi_j \ra u$ in $L^p(\Theta)$ and  $\nabla \psi_j\ra \vec{g}$ in
 $\mathcal{L}^p_Q(\Theta)$ as $j\rightarrow \infty$.
For each $j\in\mathbb{N}$, define $\phi_j = \psi_jv$.  Then $\{\phi_j\}\subset Lip_0(\Theta)$ and
\[  \|uv - \phi_j\|_{L^p(\Theta)} \leq
  \|v\|_{L^\infty(\Theta)}\|u-\psi_j\|_{L^p(\Theta)}; \]
hence, $\phi_j \rightarrow uv$ in $L^p(\Theta)$ as $j\ra\infty$.

To complete the proof, we will  show that
$\nabla \phi_j \rightarrow v\vec{g} + u\nabla v$ in
$\mathcal{L}_Q^p(\Theta)$.  Since
$\nabla \phi_j = v\nabla \psi_j + \psi_j\nabla v$, we have that
\begin{align*}
& \int_\Theta |\sqrt{Q}(v\vec{g} + u\nabla v - v\nabla \psi_j -
\psi_j\nabla v)|^p\,dx \\
& \qquad \qquad \leq
 C\bigg[\int_\Theta |\sqrt{Q}(\vec{g}-\nabla\psi_j|^p|v|^pdx
+ \int_\Theta |\sqrt{Q}\nabla v|^p|u-\psi_j|^p\,dx\bigg] \\
& \qquad \qquad \leq 
C\bigg[\|v\|^p_{L^\infty(\Theta)}\int_\Theta |\sqrt{Q}(\vec{g}-\nabla\psi_j)|^p\,dx
+ \|\sqrt{Q}\nabla
v\|_{L^{pt}(\Theta)}^p\|u-\psi_j\|_{L^{pt'}(\Theta)}^p\bigg]. 
\end{align*}
 The first term on in the last line goes to $0$ by our choice of
 $\psi_j$ and if $t'=1$, then the second term does as well.  If
 $t'>1$, then 
to estimate the  second term, note first that it follows from the
local Poincar\'e inequality~\eqref{localpc} that for all $y$ and $r>0$
sufficiently small and $f\in Lip_0(\Omega)$,
\[ \bigg(\fint_{B(y,r)} |f|^{pt'}\,dx\bigg)^{\frac{1}{pt'}}
\leq Cr\bigg(\fint_{B(y,\beta r)}|\sqrt{Q}\nabla
f|^p\,dx\bigg)^{\frac{1}{p}}
+ \bigg(\fint_{B(y,r)} |f|^{p}\,dx\bigg)^{\frac{1}{p}}
\]
Therefore, by a partition of unity argument like that used to
prove the weak global Sobolev inequality~\eqref{weakglobalsob} from
the weak local Sobolev inequality~\eqref{localsob}, we have that
\begin{equation} \label{eqn:weak-sob}
 \bigg(\int_{\Theta} |f|^{pt'}\,dx\bigg)^{\frac{1}{pt'}}
\leq C\bigg(\int_{\Theta}|\sqrt{Q}\nabla
f|^p\,dx\bigg)^{\frac{1}{p}}
+ \bigg(\int_{\Theta} |f|^{p}\,dx\bigg)^{\frac{1}{p}}.
\end{equation}
But then, in the second term we have that
$\|\sqrt{Q}\nabla v\|_{L^{pt}(\Theta)}<\infty$ since $\nabla v\in
L^\infty(\Theta)$ and $Q\in L^{pt}_{loc}(\Omega)$.  
Moreover, by \eqref{eqn:weak-sob} we have that
\[ \|u-\psi_j\|_{L^{pt'}(\Theta)} \leq
  \|u-\psi_j\|_{\widehat{H}^{1,p}_{Q}(\Theta)}, \]
and the right-hand term goes to $0$ as $j\rightarrow \infty$.
Therefore, we have shown that $(uv,v\vec{g} + u\nabla v)\in
\widehat{H}^{1,p}_{Q,0}(\Theta)$. 
\end{proof}

\medskip

\begin{proof}[Proof of Theorem~\ref{mainlocglob}]
Fix $v\in Lip_0(\Theta)$.  Our goal is to show that the global Sobolev
estimate~\eqref{sobgoal} holds.  It will be enough to show that for
some $\eta$, $1<\eta<\sigma$,
\begin{equation} \label{eqn:sob-eta}
 \bigg(\int_\Theta |v|^{p\eta}\,dx\bigg)^{\frac{1}{p\eta}}
\leq C\bigg(\int_\Theta |\sqrt{Q} \nabla
v|^p\,dx\bigg)^{\frac{1}{p}}. 
\end{equation}
For given this, by the weak global Sobolev
inequality~\eqref{weakglobalsob}, we have that
\begin{multline*}
 \bigg(\int_\Theta |v|^{p\sigma}\,dx\bigg)^{\frac{1}{p\sigma}}
 \leq  C\bigg[\bigg(\int_\Theta |\sqrt{Q}\nabla
  v|^p\,dx\bigg)^\frac{1}{p} 
+ \bigg(\int_\Theta |v|^p\,dx\bigg)^\frac{1}{p}\bigg] \\
\leq C\bigg[\bigg(\int_\Theta |\sqrt{Q}\nabla
  v|^p\,dx\bigg)^\frac{1}{p} 
+ \bigg(\int_\Theta |v|^{p\eta}dx\bigg)^\frac{1}{p\eta}\bigg]
\leq C\bigg(\int_\Theta |\sqrt{Q}\nabla v|^pdx\bigg)^\frac{1}{p},
\end{multline*}
and this is the desired inequality.

\smallskip

To prove~\eqref{eqn:sob-eta}, fix $q\in
[p',\infty)\cap(p\sigma',\infty)$ and let $\eta=q'>1$; note that 
$1<\eta<\sigma$ since $q>p\sigma'>\sigma'$.  Then by duality we have that
\begin{equation} \label{eqn:dual}
 \bigg(\int_\Theta |v|^{p\eta}\,dx\bigg)^\frac{1}{p\eta} 
= 
\bigg[\bigg(\int_\Theta (|v|^p)^\eta \,dx\bigg)^\frac{1}{\eta}\bigg]^{\frac{1}{p}}
=
\sup\bigg[ \int_\Theta \vphi |v|^p\,dx\bigg]^\frac{1}{p}, 
\end{equation}
where the supremum is taken over all non-negative $\varphi \in
L^q(\Theta)$, $\|\varphi\|_{L^q(\Theta)}=1$.  

Fix  a non-negative function $\varphi \in
L^q(\Theta)$, $\|\varphi\|_{L^q(\Theta)}=1$;  we estimate the
last integral.  Fix $\tau\in (0,1)$; the exact value of
$\tau$ will be determined below.   Since $q\geq p'$ and $\Theta$ is bounded,
$\varphi \in L^{p'}(\Theta)$, and so by Proposition~\ref{existence}, 
there exists $\uu_{\tau} = (u_\tau,\vec{g}_\tau)\in
\widehat{H}^{1,p}_{Q,0}(\Theta)$ that is a weak solution
of the Dirichlet problem~\eqref{eqn:dirichlet}.  Since $|v|^p\in Lip_0(\Theta)$, we can use
it as a test function in the definition of weak solution. This yields
\[\int_\Theta{\lan}Q\vec{g}_\tau,\vec{g}_\tau
{\ran}^\frac{p-2}{2}{\lan}Q\vec{g}_\tau,\nabla (|v|^p){\ran}\,dx 
+ \tau\int_\Theta|u_\tau|^{p-2}u_\tau |v|^p\,dx = -\int_\Theta \varphi
|v|^p\,dx. \]
If we take absolute values, rearrange terms and apply H\"older's inequality,  we get
\begin{align}
 \int_\Theta \varphi|v|^p\,dx 
& \leq  p\bigg|\int_\Theta
  {\lan}Q\vec{g}_\tau,\vec{g}_\tau{\ran}^\frac{p-2}{2}
{\lan}Q\vec{g}_\tau,\nabla v{\ran}|v|^{p-1}\sgn(v)\,dx\bigg| \nonumber\\
& \qquad \qquad + \tau\bigg|\int_\Theta|u_\tau|^{p-2}u_\tau |v|^p\,dx\bigg|\nonumber\\
& \leq p\int_\Theta
  {\lan}Q\vec{g}_\tau,\vec{g}_\tau{\ran}^\frac{p-1}{2}
{\lan}Q\nabla v,\nabla v{\ran}^\frac{1}{2}|v|^{p-1}\,dx \nonumber\\
& \qquad \qquad + \tau \|u_\tau\|_{L^\infty(\Theta)}^{p-1}\|v\|^p_{L^p(\Theta)}\nonumber\\
& \leq p\bigg( \int_\Theta
  {\lan}Q\vec{g}_\tau,\vec{g}_\tau{\ran}^\frac{p}{2}|v|^p\,dx\bigg)^{\frac{1}{p'}}
  \bigg(\int_\Theta{\lan}Q\nabla v,\nabla
       v{\ran}^\frac{p}{2}\,dx\bigg)^\frac{1}{p} \label{final01}\\
& \qquad \qquad + \tau
       \|u_\tau\|_{L^\infty(\Theta)}^{p-1}\|v\|^p_{L^p(\Theta)}. \nonumber
\end{align}

To estimate \eqref{final01}, define
\[  {\bf A} = \int_\Theta
  {\lan}Q\vec{g}_\tau,\vec{g}_\tau{\ran}^\frac{p}{2}|v|^p\,dx. \]
Let $\vec{h} = pu_\tau |v|^{p-1}\sgn(v)\nabla v + |v|^p\vec{g}_\tau$;
then by Lemma \ref{prodlem} we have
$(u_\tau |v|^p, \vec{h})\in \widehat{H}^{1,p}_{Q,0}(\Theta)$.
Moreover, we have that
\[ \int_\Theta {\lan}Q\vec{g}_\tau,\vec{g}_\tau{\ran}^\frac{p-2}{2}
{\lan}Q\vec{g}_\tau, \vec{h}{\ran}\,dx 
= {\bf A} + p\int_\Theta
{\lan}Q\vec{g}_\tau,\vec{g}_\tau{\ran}^\frac{p-2}{2}
{\lan}Q\vec{g}_\tau,\nabla v{\ran}u_\tau\sgn(v) |v|^{p-1}\,dx.
\]
Since $\uu_\tau$ is a weak solution of~\eqref{eqn:dirichlet} and
$(u_\tau |v|^p,\vec{h})$ can be used as a test function, we have that
\begin{align*} 
\qquad {\bf A}
&\leq \bigg|\int_\Theta
  {\lan}Q\vec{g}_\tau,\vec{g}_\tau{\ran}^\frac{p-2}{2}
{\lan}Q\vec{g}_\tau, \vec{h}{\ran}\,dx\bigg| 
+ p\bigg|\int_\Theta
  {\lan}Q\vec{g}_\tau,\vec{g}_\tau{\ran}^\frac{p-2}{2}
{\lan}Q\vec{g}_\tau,\nabla v{\ran}u_\tau \sgn(v)|v|^{p-1}\,dx\bigg|\nonumber\\
&\leq \bigg|\int_\Theta \varphi u_\tau |v|^pdx\bigg| 
+ \tau\int_\Theta |u_\tau|^p|v|^pdx\nonumber\\
& \qquad \qquad  +\; p \int_\Theta
  {\lan}Q\vec{g}_\tau,\vec{g}_\tau{\ran}^\frac{p-1}{2}
{\lan}Q\nabla v,\nabla v{\ran}^\frac{1}{2}|u_\tau| |v|^{p-1}\,dx\nonumber\\
&\leq \bigg|\int_\Theta \varphi u_\tau |v|^p\,dx\bigg| 
+ \tau\int_\Theta |u_\tau|^p|v|^p\,dx 
 + C(p)\int_\Theta {\lan}Q\nabla v,\nabla v{\ran}^\frac{p}{2}|u_\tau|^p
  \,dx \\
& \qquad \qquad
+\frac{1}{2} \int_\Theta
  {\lan}Q\vec{g}_\tau,\vec{g}_\tau{\ran}^\frac{p}{2}|v|^p\,dx; 
\end{align*}
the last inequality follows from Young's inequality.   The last term
equals $\frac{1}{2}\bf A$.  Moreover, since $\|\varphi\|_{L^q(\Theta)}=1$, by
Proposition~\ref{tauindep} there exists a constant $K$, independent of
$\tau$, such that
$\|u_\tau\|_{L^\infty(\Theta)}\leq K$.  Therefore, the above
inequality yields
\begin{equation} \label{abound2}
{\bf A} 
\leq C(p)\bigg[K\int_\Theta \varphi|v|^p\,dx 
+ \tau K^p
 \|v\|_{L^p(\Theta)}^p +K^p
\int_\Theta{\lan}Q\nabla v,\nabla v{\ran}^\frac{p}{2}\,dx\bigg]. 
\end{equation}

Irrespective of which term on the right-hand side of \eqref{abound2}
is the maximum, if we combine \eqref{final01} with \eqref{abound2} we get
\begin{equation} \label{finalgrail} 
\int_\Theta \varphi |v|^p\,dx 
\leq C(p,K) \bigg[ \int_\Theta {\lan}Q\nabla v,\nabla
v{\ran}^\frac{p}{2}\,dx 
+ \tau  \|v\|_{L^p(\Theta)}^p\bigg].
\end{equation}
To see that this is true, suppose first that 
the largest term is
$\tau K^p\|v\|_{L^p(\Theta)}^p$.  Then by
Young's inequality we
have that
\begin{align*}
\int_\Theta \varphi |v|^p\,dx 
&\leq C\tau^\frac{p-1}{p}K^{p-1}\|v\|_{L^p(\Theta)}^{p-1}
\bigg(\int_\Theta {\lan}Q\nabla v,\nabla
  v{\ran}^\frac{p}{2}\,dx\bigg)^\frac{1}{p}
+ \tau K^{p-1}\|v\|_{L^p(\Theta)}^p\\
&= CK^{p-1}\bigg[\tau^\frac{p-1}{p}\|v\|_{L^p(\Theta)}^{p-1}
\bigg(\int_\Theta {\lan}Q\nabla v,\nabla
    v{\ran}^\frac{p}{2}\,dx\bigg)^\frac{1}{p} 
+ \tau \|v\|_{L^p(\Theta)}^p\bigg]\\
&\leq CK^{p-1}\bigg[\int_\Theta {\lan}Q\nabla v,\nabla
  v{\ran}^\frac{p}{2}\,dx
 + \tau\|v\|_{L^p(\Theta)}^p\bigg].
\end{align*}
The other estimates are proved similarly.

Given inequality~\eqref{finalgrail} it is now straightforward to prove
the desired estimate:
\begin{align*}
\bigg(\int_\Theta \vphi |v|^p\,dx\bigg)^\frac{1}{p}
&\leq C(p,K)^\frac{1}{p}\bigg[\int_\Theta{\lan}Q\nabla v,\nabla
v{\ran}^\frac{p}{2}\,dx 
+ \tau\|v\|_{L^p(\Theta)}^p\bigg]^\frac{1}{p} \\
&\leq C(p,K) ^\frac{1}{p}
\bigg[\bigg(\int_\Theta{\lan}Q\nabla v,\nabla
  v{\ran}^\frac{p}{2}\,dx\bigg)^\frac{1}{p} 
+ \tau^\frac{1}{p}\bigg(\int_\Theta |v|^{p\eta}dx\bigg)^\frac{1}{p\eta}\bigg].
\end{align*}
Fix $\tau<1$ such that $\tau C(p,K) \leq \frac{1}{2}$.  Since this
constant is independent of $\varphi$, if we combine this inequality with the
duality estimate~\eqref{eqn:dual}, we get that
\[ \bigg(\int_\Theta |v|^{p\eta}\,dx\bigg)^{\frac{1}{p\eta}}
\leq C(p,K)\bigg(\int_\Theta{\lan}Q\nabla v,\nabla
  v{\ran}^\frac{p}{2}\,dx\bigg)^\frac{1}{p} 
+ \frac{1}{2}\bigg(\int_\Theta |v|^{p\eta}dx\bigg)^\frac{1}{p\eta}. \]
If we re-arrange terms we get \eqref{eqn:sob-eta} and our proof is complete.
\end{proof}


\bibliographystyle{plain}

\end{document}